\documentclass[11pt]{amsart}
\pdfoutput=1
\usepackage{amssymb}
\usepackage{amsmath}
\usepackage{amsfonts}
\usepackage[latin1]{inputenc}
\usepackage[french]{babel}
\usepackage{multirow}
\usepackage{graphicx}
\usepackage{enumerate}

\newtheorem{theoreme}{Théorème}[section]
\newtheorem{lemme}[theoreme]{Lemme}

\newtheorem{corollaire}[theoreme]{Corollaire}
\newtheorem{conjecture}[theoreme]{Conjecture}

\theoremstyle{remark}
\newtheorem{remarque}{\it Remarque}[section]
\numberwithin{equation}{section}

\newcommand{\A}{\mathcal{A}}
\newcommand{\B}{\mathcal{B}}

\newcommand{\bbR}{\mathbb{R}}
\newcommand{\bbZ}{\mathbb{Z}}

\title{Dynamique du problème $3x+1$ sur la droite réelle}
\author{Nik Lygeros et Olivier Rozier}
\address{LGPC (UMR 5285), Université de Lyon, 69616, Villeurbanne, France}
\email{nlygeros@gmail.com}
\address{IPGP (UMR 7154), 75238, Paris, France}
\email{olivier.rozier@gmail.com}

\begin{document}

\sloppy 

\maketitle

\begin{abstract}
Le \textit{problème $3x+1$} est une conjecture réputée difficile au sujet d'un algorithme pourtant très simple sur les entiers positifs. Une approche possible est de sortir du cadre discret. M. Chamberland a ainsi utilisé une extension analytique sur la demi-droite $\mathbb{R}^{+}$. Nous complétons ses résultats sur la dynamique des points critiques et obtenons une nouvelle formulation du problème $3x+1$. Nous explicitons alors les liens avec la question de l'existence d'intervalles errants. Puis nous étendons l'étude de la dynamique à la demi-droite $\mathbb{R}^{-}$, en relation avec le \textit{problème $3x-1$}. Enfin nous analysons le comportement moyen des itérations réelles au voisinage de $\pm \infty$ et présentons une méthode basée sur un argument heuristique de R. E. Crandall dans le cas discret. Il suit que la vitesse moyenne des itérations réelles est proche de $ (2+\sqrt{3})/4 $ sous condition de distribution uniforme modulo 2. 

\end{abstract}

\renewcommand\abstractname{Abstract}
\begin{abstract}
The $3x+1$ \textit{problem} is a difficult conjecture dealing with quite a simple algorithm on the positive integers. A possible approach is to go beyond the discrete nature of the problem, following M. Chamberland who used an analytic extension to the half-line $\mathbb{R}^{+}$. We complete his results on the dynamic of the critical points and obtain a new formulation the $3x+1$ problem. We clarify the links with the question of the existence of wandering intervals. Then, we extend the study of the dynamic to the half-line $\mathbb{R}^{-}$, in connection with the $3x-1$ \textit{problem}. Finally, we analyze the mean behaviour of real iterations near $\pm \infty$ and  present a method based on a heuristic argument by R. E. Crandall in the discrete case. It follows that the average growth rate of the iterates is close to $(2+\sqrt{3})/4$ under a condition of uniform distribution modulo 2. 
\end{abstract}

\section{Introduction}
\label{intro}

Généralement attribué à Lothar Collatz, le \textit{problème $3x+1$} est aussi appelé  \textit{conjecture de Syracuse}, en référence à l'Université du même nom. Il se rapporte à la fonction $T$ définie sur les entiers positifs par
\begin{equation}
T(n) := \left\{\begin{array}{cl}
  (3n+1)/2 & \mbox{si n est impair,} \\
  n/2 & \mbox{sinon.} 
\end{array}\right.
\end{equation}

Il s'agit de prouver que toute itération de $T$ à partir d'un entier positif $n$ arbitraire conduit nécessairement à la valeur 1. Cette valeur est cyclique de période 2 : $T(T(1)) = 1$.

\begin{conjecture}\label{conj_T} \textbf{Problème $3x+1$}\\
Pour tout entier $n > 0$,  il existe un entier $k \geq 0$ tel que $T^{k}(n) = 1$.\footnote{On note $T^{k}(n)$ le $k^{\text{ème}}$ itéré de $T$.}
\end{conjecture}

La figure \ref{fig: inv_tree} représente toutes les orbites qui aboutissent à 1 en un maximum de sept itérations.\\

\begin{figure}[t]
\centering
\includegraphics[scale=0.5]{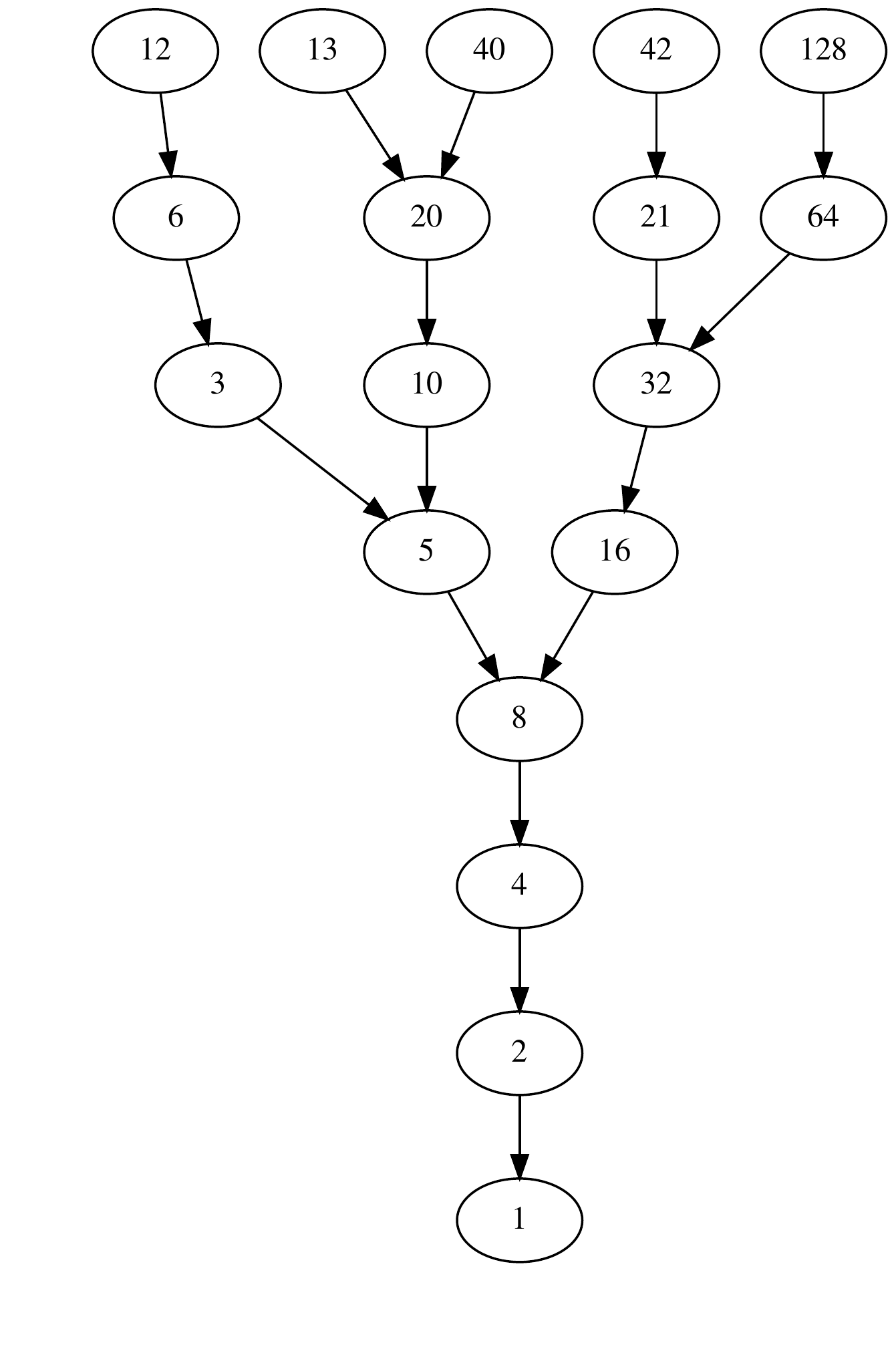}
\caption{Arbre inverse du problème $3x+1$ représentant l'ensemble des antécédents de 1 sur sept itérations.}
\label{fig: inv_tree}
\end{figure}

Le problème $3x+1$ se ramène entièrement aux deux conjectures \ref{conj_T1} et \ref{conj_T2} sur la dynamique de la fonction $T$.

\begin{conjecture}\label{conj_T1} \textbf{Absence de trajectoires divergentes}\\
Tout entier positif $n$ a une orbite $\left\lbrace   T^{i}(n) \right\rbrace  _{i=0}^{\infty}$ bornée.
\end{conjecture}

\begin{conjecture}\label{conj_T2} \textbf{Absence de cycles non-triviaux}\\
Il n'existe pas d'entiers $n > 2$ et $k > 0$ tels que $T^{k}(n) = n$.
\end{conjecture}

La conjecture \ref{conj_T1} implique que tout entier positif a une orbite cyclique à partir d'un certain rang par itération de $T$. La conjecture \ref{conj_T2} stipule que le seul cycle possible est le cycle $(1,2)$.

Généralement, on convient de stopper les itérations lorsque la valeur 1 est atteinte. Ainsi on appelle \textit{temps de vol} de $n$ le plus petit entier $k$ tel que $T^{k}(n) = 1$.

T. Oliveira e Silva a vérifié par des calculs sur ordinateur que tout entier positif $n<5 \cdot 2^{60}$ a un temps de vol fini \cite{Lag10,RooWeb}.

Les conjectures \ref{conj_T1} et \ref{conj_T2}, bien qu'abondamment étudiées, ne sont toujours pas résolues. On pourra se référer aux ouvrages de J. Lagarias \cite{Lag10} et G.J. Wirsching \cite{Wir98} pour une synthèse détaillée des résultats partiels relatifs au problème $3x+1$ et diverses variantes.

R. E. Crandall \cite{Cra78} a avancé un argument heuristique basé sur l'idée de promenade aléatoire : si l'on considère uniquement la sous-suite des itérés impairs d'un entier $n$ assez grand, on s'attend à ce que l'ensemble des rapports possibles entre deux termes successifs impairs, à savoir 3/2, 3/4, 3/8, $\ldots$, aient pour probabilités respectives les valeurs 1/2, 1/4, 1/8, $\ldots$. On obtient comme rapport moyen la valeur $3/4$. Ceci découle de l'égalité
\begin{equation}
\left( \frac{3}{2}\right)^{\frac{1}{2}} \cdot \left( \frac{3}{4}\right)^{\frac{1}{4}} \cdot \left( \frac{3}{8}\right)^{\frac{1}{8}} \cdots = \frac{3}{4}.
\end{equation}
Cet argument plaide fortement en faveur de la conjecture \ref{conj_T1}.

Dans le cadre de notre étude, nous appellerons \textit{vitesse moyenne} d'une séquence finie $\left\lbrace  n, T(n), \cdots, T^{k}(n) \right\rbrace $ la quantité $\left( T^{k}(n)/n \right) ^{1/k}$.

Un raisonnement analogue \cite{Bar98} à celui de Crandall suggère que la vitesse moyenne d'une séquence arbitraire non-cyclique a statistiquement une valeur proche de $\sqrt{3}/2 \simeq 0.866 \ldots$, moyenne géométrique de $1/2$ et $3/2$. En effet, la croissance d'une séquence dépend principalement de la parité des itérés successifs. Or, on s'attend à ce que les parités soient équiréparties sur un grand nombre d'itérations.

Ainsi le temps de vol $k$ d'un entier $n$ serait tel que $(1/n)^{1/k} \approx \sqrt{3}/2$ et l'on obtiendrait la valeur moyenne 
 $$ k \approx \frac{2\ln n}{\ln \left( \frac{4}{3}\right)} $$
en l'absence de cycle \cite[p. 7]{Lag10}. 

Ces estimations sont confortées par les calculs numériques. Il semble donc qu'un tel raisonnement permette de saisir l'essentiel de la dynamique asymptotique du problème $3x+1$.\\

\section{Extension sur les réels positifs}
Une approche possible du problème $3x+1$ est de sortir du cadre discret et d'étendre $T$ par une fonction analytique sur l'ensemble des nombres réels \cite{Cha96} ou complexes \cite{Rei03, Sch99}. Nous opterons pour l'extension réelle\footnote{Le deuxième auteur (O. Rozier) avait antérieurement suggéré l'étude de l'extension \eqref{eq_ext} dans le plan complexe et obtenu des représentations graphiques des bassins d'attraction \cite{Roz90}.}  qui nous parait la plus naturelle, définie par l'équation \eqref{eq_ext} ci-après, et nous expliciterons les liens étroits qu'entretiennent la dynamique sur les réels et le problème $3x+1$.

Chamberland \cite{Cha96} a étudié la dynamique sur la demi-droite $\mathbb{R}^{+}$ de la fonction analytique 
\begin{equation}\label{eq_ext}
f(x) := x+\frac{1}{4}-\left(\frac{x}{2}+\frac{1}{4}\right)\cos(\pi x)
\end{equation}
qui vérifie $f(n) = T(n)$ pour tout entier $n$, et $f\left(\mathbb{R}^{+}\right) = \mathbb{R}^{+}$ . Il a ainsi obtenu plusieurs résultats significatifs :

\begin{align}
\label{prop_cycles}
& \text{Le point fixe 0 est attractif ainsi que les cycles $\A_{1} := \lbrace 1,2 \rbrace$ et}\\
& \text{$\A_{2} := \lbrace 1.192\ldots,  2.138\ldots\rbrace$ de période 2.} \notag \\
\label{prop_schwartz}
& \text{La dérivée \textit{Schwartzienne} de $f$ est négative sur $\mathbb{R}^{+}$}. \\
\label{prop_int_inv}
& \text{Les intervalles $[0, \mu_{1}]$ et $[\mu_{1}, \mu_{3}]$ sont invariants par $f$, où}  \\ 
& \text{$\mu_{1}=0.277\ldots$ et $\mu_{3}=2.445\ldots$ sont des points fixes répulsifs.} \notag \\
\label{prop_int_cycles}
& \text{Tout cycle d'entiers positifs est attractif.} \\
\label{prop_unbound}
& \text{Il existe des orbites monotones non-bornées sur $\mathbb{R}^{+}$.}
\end{align}

Par ailleurs, il énonce la conjecture ``Stable Set" \cite{Cha96} ci-dessous :
\begin{conjecture}\label{conj_cycles} \textbf{Cycles attractifs sur $\mathbb{R}^{+}$}\\
La fonction $f$ n'admet aucun cycle attractif sur l'intervalle $[\mu_{3}, +\infty)$.
\end{conjecture}

Une conséquence immédiate de \eqref{prop_int_cycles} est que la conjecture \ref{conj_cycles} entraîne la conjecture \ref{conj_T2} du problème $3x+1$.\\

Puis, il définit l'ensemble des orbites non-bornées
\begin{equation}
U^{\infty}_{f} := \left\lbrace x \in \mathbb{R}^{+} : \limsup_{k \rightarrow \infty} f^{k}(x) = \infty \right\rbrace.
\end{equation}
Le résultat \eqref{prop_unbound} prouve que $U^{\infty}_{f}$ est infini, et l'on démontre que $U^{\infty}_{f}$ contient un ensemble de Cantor dans chaque intervalle $[n,n+1]$ pour tout entier $n \geq 2$ \cite{Sch99}. Il suit que $U^{\infty}_{f}$ n'est pas dénombrable.

\begin{conjecture}\label{conj_unbound} \textbf{Orbites non-bornées sur $\mathbb{R}^{+}$}\\
L'ensemble $U^{\infty}_{f}$ est d'intérieur vide.
\end{conjecture}
La conjecture \ref{conj_unbound} est une formulation faible de la conjecture ``Unstable Set" \cite{Cha96}. Nous allons montrer qu'elle a des liens logiques avec le problème $3x+1$.

\begin{lemme}\label{lem_critic}
Soit $\left\lbrace c_{n} \right\rbrace _{n=0}^{\infty}$ l'ensemble des points critiques de $f$ dans $\mathbb{R}^{+}$, ordonnés de telle sorte que $ 0 < c_1{} < c_{2} < \ldots $.\\
Alors on a
$$ n - \frac{1}{\pi^{2}n} < c_{n} < n  \text{, si $n$ est pair;}$$
$$ n < c_{n} < n + \frac{3}{\pi^{2}n} \text{, si $n$ est impair.}$$
\end{lemme}
\begin{proof}(indications)
Soit $n$ un entier positif. 
On a $$f'(x) = 1 - \frac{1}{2} \cos(\pi x) + \pi \left(\frac{x}{2}+\frac{1}{4}\right)\sin(\pi x)$$
et on vérifie facilement que 
$n - \frac{1}{2} < c_{n} < n$ si $n$ est pair, et
$n < c_{n} < n + \frac{1}{2}$ si $n$ est impair.

De plus, on a toujours $f'(n) > 0$ et on montre que
$$f'\left( n-\frac{1}{\pi^{2}n}\right) < \frac{ \left( 20-6\pi^{2}n\right) n+1}{24\pi^{2}n^{3}} < 0 \text{, si $n$ est pair,}$$
$$f'\left( n + \frac{3}{\pi^{2}n}\right) < \frac{\left( 18-6\pi^{2}n\right) n+9}{8\pi^{2}n^{3}} < 0 \text{, si $n$ est impair,}$$
en utilisant les encadrements $ 1 - \frac{t^{2}}{2} < \cos t  < 1$ et $ t - \frac{t^{3}}{6} < \sin t < t$ pour $0 < t < 1$.
\end{proof}

\begin{lemme}\label{lem_stable}
On considère la famille d'intervalles $J^{a}_{n} := \left[ n , n + \frac{a}{\pi^{2}n} \right]$ pour tout entier $n > 0$ et tout réel $a$ tel que $\frac{27}{8} < a < 6$. \\
Alors on a $ f \left( J^{a}_{n} \right) \subset J^{a}_{f(n)}$ pour tout entier $n$ assez grand.\\
Si de plus $a=\frac{7}{2}$, alors l'inclusion est vraie pour tout $n>0$.
\end{lemme}
\begin{proof}
Soit un entier $n>0$ et un réel $a$ tel que $\frac{27}{8} < a < 6$.\\

1\ier{} cas : $n$ est pair, $f(n) = \frac{n}{2}$ et $f$ est croissante sur $J^{a}_{n}$.
On vérifie alors que 
$$ f\left( n + \frac{a}{\pi^{2}n}\right) \leq f(n) + \frac{a}{\pi^{2}f(n)} + A \cdot B$$
avec $$A = \frac{a}{8\pi^{4}n^{3}} \text{ et } B = \pi^{2}n\left( 2\left( a - 6 \right) n + a \right)  + 2a^{2}$$
en utilisant l'inégalité $1 - \cos t  < \frac{t^{2}}{2}$ pour $0 < t < 1$. Comme $a - 6 < 0$, il est clair que $A \cdot B <0$ pour $n$ suffisamment grand.\\

Si de plus $a = \frac{7}{2}$, alors $B \leq \frac{49}{2}-13\pi^{2} < 0$ pour tout $n$.\\

2\ieme{} cas : $n$ est impair, $f(n) = \frac{3n+1}{2}$ et $f$ est croissante sur $\left[ n, c_{n} \right]$ et décroissante sur $\left[ c_{n}, n + \frac{a}{\pi^{2}n} \right]$.
On vérifie alors que
$$ f\left( n + \frac{a}{\pi^{2}n}\right) \geq f(n) - A \cdot B$$
$A$ et $B$ étant défini comme précédemment, donc $A \cdot B < 0$ pour $n$ suffisamment grand. Si de plus $a = \frac{7}{2}$, alors $A \cdot B < 0$ pour tout $n \geq 3$, et dans le cas $n=1$, on a 
$$f\left( 1 + \frac{7}{2 \pi^{2}} \right) = 2.013 \ldots > f(1).$$

D'après le lemme \ref{lem_critic}, on a $c_{n} = n + \frac{b}{\pi^{2}n}$ avec $0 < b < 3$. Il vient
$$f(c_{n}) - f(n) \leq \frac{3b}{2\pi^{2}n} - \frac{n}{2}\left( 1- \cos\left( \frac{b}{\pi n} \right) \right)$$
puis en utilisant l'inégalité $1 - \cos t > \frac{t^{2}}{2} - \frac{t^{4}}{24}$ pour $0 < t < 1$,
$$f(c_{n}) - f(n) < \frac{b(6 - b)}{4\pi^{2}n} + \frac{b^{4}}{48\pi^{4}n^{3}} \leq \frac{9}{4\pi^{2}n} + \frac{27}{16\pi^{4}n^{3}}.$$

On obtient
$$ f(c_{n}) < f(n) + \frac{a}{\pi^{2}f(n)} + \frac{C}{D}$$
avec 
$$C = 4\pi^{2}n^{2}\left( (27-8a)n+9\right) + 81n + 27 \text{ et } D = 16\pi^{4}n^{3}(3n+1).$$ On voit que $C < 0$ pour $n$ suffisamment grand. Si de plus $a = \frac{7}{2}$ et $n \geq 11$, on a alors 
$$C = 4\pi^{2}n^{2}(9-n) + 81n + 27 < 0$$
 et dans les cas où $n$ = 1, 3, 5, 7 ou 9, on vérifie numériquement que 
 $$f(c_{n}) - f(n) - \frac{7}{(3n+1)\pi^{2}} < 0$$
 en utilisant les valeurs $c_{1} = 1.180938\ldots$, $c_{3} = 3.084794...$, $c_{5} = 5.054721...$, $c_{7} = 7.040311...$ et $c_{9} = 9.031889...$.

\end{proof}

On déduit du lemme \ref{lem_stable} un lien logique entre les conjectures \ref{conj_T1} et \ref{conj_unbound} :
\begin{theoreme}\label{th_unbound}
La conjecture \ref{conj_unbound} implique la conjecture \ref{conj_T1} (absence d'orbites non-bornées) du problème $3x+1$ .
\end{theoreme}

\begin{proof}
Supposons que la conjecture \ref{conj_unbound} soit vraie et que la conjecture \ref{conj_T1} soit fausse. Alors il  existe un entier positif $n_{0}$ tel que $\limsup_{k \rightarrow \infty} f^{k}(n_{0}) = \infty$. D'après le lemme \ref{lem_stable}, une simple récurrence donne $f^{k}\left( J^{\frac{7}{2}}_{n_{0}} \right) \subset J^{\frac{7}{2}}_{f^{k}(n_{0})}$ pour tout entier $k \geq 0$. Donc l'ensemble $U^{\infty}_{f}$ contient l'intervalle $J^{\frac{7}{2}}_{n_{0}}$, ce qui est en contradication avec notre hypothèse que $U^{\infty}_{f}$ soit d'intérieur vide.
\end{proof}

\section{Dynamique des points critiques}
\label{critic}

Les résultats \eqref{prop_schwartz} et \eqref{prop_int_cycles} entrainent que le bassin d'attraction immédiat de tout cycle d'entiers strictement positifs contient au moins un point critique \cite{Cha96}. Pour cette raison, Chamberland a effectué des calculs numériques relatifs aux orbites des points critiques $c_{n}$ pour $n \leq 1000$. Il énonce la conjecture ``Critical Points" ci-dessous :
\begin{conjecture}\label{conj_critical} \textbf{Points critiques}\\
Tous les points critiques $c_{n}$, $n>0$, sont attirés par l'un des cycles $\A_{1} $ ou $\A_{2}$.
\end{conjecture}

Nous complétons ici les résultats numériques de Chamberland. Une précision de 1500 chiffres décimaux en virgule flottante est requise pour le calcul de certaines orbites ($c_{646}$ par exemple). Nous avons vérifié nos résultats avec deux logiciels différents, Mathematica et Maple.

D'après nos calculs, les cycles $\A_{1}$ et $\A_{2} $ attirent tous les points critiques $c_{n}$ pour $n \leq 2000$. Plus précisément, $c_{n}$ est attiré par $\A_{2}$ pour \footnote{En gras les valeurs déjà obtenues par Chamberland.} $n$ = \textbf{1}, \textbf{3}, \textbf{5}, 382, 496, \textbf{502}, 504, 508, 530, 550, 644, 646, \textbf{656}, 666, 754, 830, 874, 1078, 1150, 1214, 1534, 1590, 1598, 1614, 1662, 1854, et  par $\A_{1} $ pour toutes les autres valeurs de $n \leq 2000$. Nous avons observé que l'orbite de $c_{n}$ est toujours proche de l'orbite de $n$, sauf pour $n \equiv -2 \pmod{64}$ et pour $n$=54, 334, 338, 366, 390, 442, 444, 470, 484, 486, 496, 500, $\ldots$.

Les résultats numériques suggèrent la conjecture suivante\footnote{Dans \cite{Rei03}, une conjecture analogue avec davantage d'hypothèses est formulée relativement à une autre extension de la fonction $T$ sur les réels.} :

\begin{conjecture}\label{conj_odd_critical} \textbf{Points critiques d'ordre impair}\\
Les points critiques $c_{n}$ sont attirés par le cycle $\A_{1} = \left\lbrace 1, 2 \right\rbrace $ pour tout entier $n \geq 7$ impair.
\end{conjecture}

Nous montrons à présent que la conjecture \ref{conj_odd_critical} suffit pour reformuler complètement le problème $3x+1$.

\begin{theoreme}\label{th_odd_critical}
Soit un entier impair $n \geq 7$ dont l'orbite contient 1. Alors le point critique $c_{n}$ est attiré par le cycle $\A_{1}$.
\end{theoreme}
\begin{proof}
Considérons un entier impair $n \geq 7$ dont l'orbite contient 1. La construction de l'arbre des orbites inverses de 1, représenté sur la figure \ref{fig: inv_tree}, montre que l'orbite de $n$ contient l'un des entiers 12, 13, 16 ou 40. On déduit de règles itératives modulo 3 sur les entiers que les antécédents de 12 sont des entiers pairs. Il vient que $f^{k}(n) =$ 13, 16 ou 40 pour un entier $k \geq 0$. Les lemmes \ref{lem_critic} et \ref{lem_stable} entraînent que $c_{n}$  appartient à $J^{\frac{7}{2}}_{n}$ et $f^{k}(c_{n})$ se trouve dans $J^{\frac{7}{2}}_{13} \cup J^{\frac{7}{2}}_{16} \cup J^{\frac{7}{2}}_{40}$.\\

1\ier{} cas : $f^{k}(n) = 13$, $f^{k}(c_{n}) \in J^{\frac{7}{2}}_{13}$. La séquence des itérés de $f^{k}(n)$ est $13 \rightarrow 20 \rightarrow 10 \rightarrow 5 \rightarrow 8 \rightarrow 4 \rightarrow 2 \rightarrow 1$.

Soit $m$ un entier pris dans cette séquence. La fonction $f$ est unimodale sur $J^{\frac{7}{2}}_{m}$ avec un maximum en $c_{m}$ lorsque $m$ est impair, et strictement croissante lorsque $m$ est pair. Ce comportement permet de déterminer les images successives de $J^{\frac{7}{2}}_{13}$ en fonction de $c_{13} = 13.022478\ldots$.
$$f\left(  J^{\frac{7}{2}}_{13} \right) = \left[ 20 , f(c_{13}) \right]$$
$$f^{3}\left(  J^{\frac{7}{2}}_{13} \right) = \left[ 5 , f^{3}(c_{13}) \right]$$
avec $f^{3}(c_{13}) = 5.0249\ldots < c_{5} = 5.0547\ldots$.
$$f^{7}\left(  J^{\frac{7}{2}}_{13} \right) = \left[ 1 , f^{7}(c_{13}) \right]$$ 
avec $f^{7}\left( c_{13}\right)  = 1.0184\ldots$.

De plus la fonction $f^{2}$ est strictement croissante sur l'intervalle $(1, c_{1})$ avec une unique point fixe $x_{1} = 1.023686\ldots$ qui est répulsif. Il suit que l'intervalle $[1, x_{1})$ fait partie du bassin d'attraction immédiat du cycle $\A_{1}$ et que $c_{n}$ est attiré par $\A_{1}$.\\

2\ieme{} cas : $f^{k}(n) = 16$, $f^{k}(c_{n}) \in J^{\frac{7}{2}}_{16}$. On a la séquence $16 \rightarrow 8 \rightarrow 4 \rightarrow 2 \rightarrow 1$.
Comme précédemment, on obtient l'image 
$$f^{4} \left(  J^{\frac{7}{2}}_{16} \right) = \left[ 1 , f^{4}\left( 16 + \frac{7}{32\pi^{2}}\right)  \right]$$
avec $f^{4}\left( 16 + \frac{7}{32\pi^{2}}\right)  = 1.0227\ldots <  x_{1}$. Donc $c_{n}$ est attiré par $\A_{1}$.\\

3\ieme{} cas : $f^{k}(n) = 40$, $f^{k}(c_{n}) \in J^{\frac{7}{2}}_{40}$, et la séquence des itérés est $40 \rightarrow 20 \rightarrow 10 \rightarrow 5 \rightarrow 8 \rightarrow 4 \rightarrow 2 \rightarrow 1$.
De la même manière, on itère les images successives
$$f^{3}\left(  J^{\frac{7}{2}}_{40} \right) = \left[ 5 , f^{3}\left( 40 + \frac{7}{80\pi^{2}}\right) \right] $$ avec $f^{3}\left( 40 + \frac{7}{80\pi^{2}}\right) = 5.0118\ldots < c_{5} = 5.0547\ldots$,
$$f^{7}\left(  J^{\frac{7}{2}}_{40} \right) = \left[ 1 , f^{7}\left( 40 + \frac{7}{80\pi^{2}}\right) \right]$$ avec $f^{7}\left( 40 + \frac{7}{80\pi^{2}}\right) = 1.0047\ldots < x_{1}$. Ainsi $c_{n}$ est attiré par $\A_{1}$ dans tous les cas.
\end{proof}

\begin{remarque}
Dans cette démonstration, il n'est pas possible de fusionner les cas 1 et 3 en partant de l'entier 20 car $f^{6}\left(  J^{\frac{7}{2}}_{20} \right) = \left[ 1 , f^{6}\left( 20 + \frac{7}{40\pi^{2}}\right)  \right] = \left[ 1 , 1.023691\ldots \right]$ n'est pas inclus (de très peu) dans le bassin d'attraction de $\A_{1}$ délimité par $x_{1} = 1.023686\ldots$.
\end{remarque}

\begin{corollaire}
La conjecture \ref{conj_odd_critical} est logiquement équivalente au problème $3x+1$.
\end{corollaire}
\begin{proof} Une conséquence immédiate du théorème \ref{th_odd_critical} est que la conjecture \ref{conj_T} (problème $3x+1$) implique la conjecture \ref{conj_odd_critical} sur la dynamique des points critiques d'ordre impair. On démontre à présent la réciproque.

Considérons un entier $n>0$. Son orbite contient au moins un entier impair $f^{k_{1}}(n)$, $k_{1} \ge 0$. Si $f^{k_{1}}(n) \leq 5$, alors l'orbite de $n$ contient le point 1 (cf. figure \ref{fig: inv_tree}). On considère à présent le cas $f^{k_{1}}(n) \geq 7$.

Supposons que la conjecture \ref{conj_odd_critical} soit vraie. Alors il existe un entier positif $k_{2}$ tel que 
$$f^{k_{2}}\left( c_{f^{k_{1}}(n)} \right) < 2.$$
De plus, le lemme \ref{lem_stable} donne par récurrence l'inclusion
$$f^{k_{2}}\left( c_{f^{k_{1}}(n)} \right) \in J^{\frac{7}{2}}_{f^{k_{1}+k_{2}}(n)}.$$
Il découle l'égalité
$$f^{k_{1}+k_{2}}(n) = 1.$$
\end{proof}

\section{Intervalles errants}
\label{wander}

L'existence d'\textit{intervalles errants} \cite{Mel94} dans la dynamique de l'extension $f$ est une question ouverte avec d'importantes implications pour le problème $3x+1$.
\begin{conjecture} \label{conj_wander} \textbf{Absence d'intervalles errants} \\
La fonction $f$ n'admet pas d'intervalles errants dans $\bbR^{+}$.
\end{conjecture}

Elle est au c$\oe$ur du théorème ci-dessous.
\begin{theoreme}
\label{th_wander}

On a les relations suivantes entre conjectures :\\
(a) la conjecture \ref{conj_unbound} entraîne la conjecture \ref{conj_wander},\\
(b) la conjecture \ref{conj_wander} entraîne la conjecture \ref{conj_T1}.
\end{theoreme}

\begin{proof}
Par l'absurde.\\
(a) Supposons que la conjecture \ref{conj_unbound} soit vraie et que la conjecture \ref{conj_wander} soit fausse. Cela implique que la fonction $f$ admette une famille d'intervalles errants sur une partie bornée de $\bbR^{+}$. Or ce serait en contradiction avec la propriété \eqref{prop_schwartz} : la dérivée Schwartzienne de $f$ est négative sur $\bbR^{+}$. \\

\noindent
(b) Supposons que la conjecture \ref{conj_T1} soit fausse. Alors il existe un entier positif $n$ tel que $\lim_{i \to \infty} f^{i}(n) = +\infty$. D'après le lemme \ref{lem_stable}, les intervalles $\left\lbrace f^{i}\left( J_{n}^{7/2} \right)  \right\rbrace _{i=0}^{\infty}$ sont inclus dans les intervalles $\left\lbrace J_{f^{i}(n)}^{7/2} \right\rbrace _{i=0}^{\infty}$, deux à deux disjoints. Il s'agit d'une famille d'intervalles errants.
\end{proof}

Une synthèse des liens logiques entre conjectures est donnée en figure \ref{fig: conjectures}.

\begin{figure}[ht]
\centering
\includegraphics[scale=0.75]{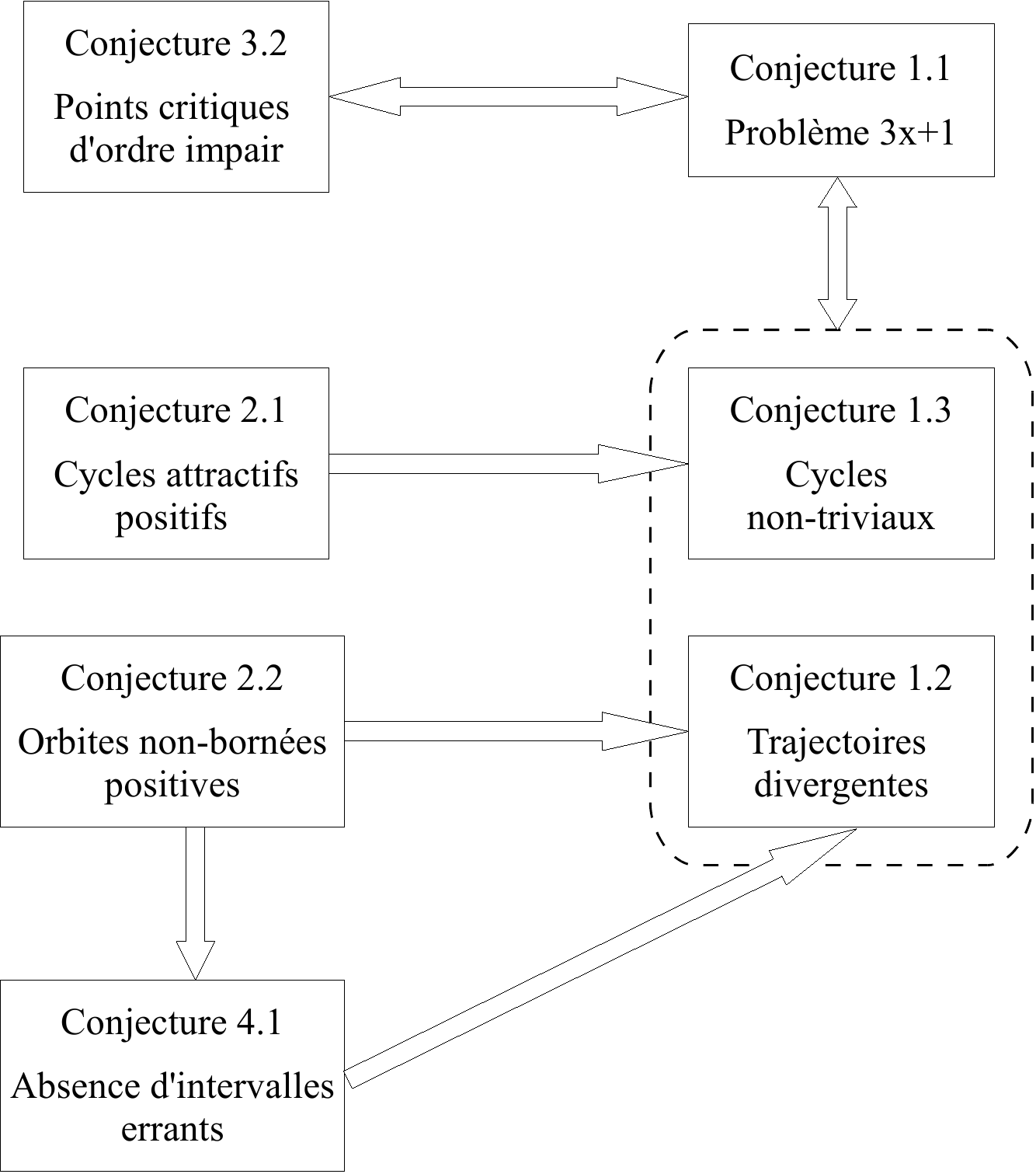}
\caption{Liens logiques entre conjectures. La partie gauche concerne le cadre continu $\bbR^{+}$ et la partie droite le cadre discret $\bbZ^{+}$.}
\label{fig: conjectures}
\end{figure}
\section{Extension sur les réels négatifs}
\label{extension_neg}

L'ensemble $\mathbb{R}^{-}$ des réels négatifs est également invariant par la fonction $f$ définie par \eqref{eq_ext}. La dynamique sur les entiers négatifs est alors identique, au signe près, à celle de la fonction ``$3x-1$'', notée $U$ et définie sur les entiers positifs par

\begin{equation}
U(n) := \left\{\begin{array}{cl}
  (3n-1)/2 & \mbox{si n est impair,} \\
  n/2 & \mbox{sinon.} 
\end{array}\right.
\end{equation}

En effet, on a la relation de conjugaison $f(-n) = -U(n)$ pour tout entier $n$ positif. La fonction $U$ admet le point fixe 1 et a deux cycles connus : $\lbrace 5, 7, 10\rbrace$ de période 3 et $\lbrace 17, 25, 37, 55, 82, 41, 61, 91, 136, 68, 34 \rbrace$ de période 11.
Cela conduit à formuler le \textit{``problème $3x-1$''} :
\begin{conjecture}\label{conj_U} \textbf{Problème $3x-1$}\\
Pour tout entier $n > 0$,  il existe un entier $k \geq 0$ tel que $U^{k}(n) = 1, 5$ ou $17$.
\end{conjecture}

Les valeurs de $f$ sur $\mathbb{R}^{+}$ et $(-\infty, -1]$ sont liées pas l'équation fonctionnelle
\begin{equation}\label{eq_fonct}
f(x) - f(-1-x) = 2x+1
\end{equation}
de sorte que les points fixes de $f$ sur $(-\infty, -1]$ sont exactement les points $\nu_{i} := -1-\mu_{i}$, où $\lbrace \mu_{i} \rbrace_{i=0}^{\infty}$ désigne l'ensemble des points fixes de $f$ sur $\mathbb{R}^{+}$, $\mu_{0} = 0 < \mu_{1} < 1 < \mu_{2} < 2 < \ldots$.\\

Néanmoins, la dynamique de $f$ sur $\mathbb{R}^{-}$ diffère partiellement de celle que l'on a pu décrire sur $\mathbb{R}^{+}$, comme le montrent les propriétés \eqref{prop_cycles_neg} à \eqref{prop_unbound_neg}.
\begin{align}
\label{prop_cycles_neg}
& \text{Les points fixes 0 et $\nu_{1}=-1.277\ldots$ sont attractifs, ainsi que les} \\
& \text{cycles } \notag \\
& \notag \B_{1} := \left\lbrace  x, f(x), f^{2}(x)\right\rbrace \text{ où } x = -5.046002\ldots, \\
 \notag & \B_{2} := \left\lbrace  x, f(x), f^{2}(x)\right\rbrace \text{ où } x = -4.998739\ldots, \\
 \notag & \B_{3} := \left\lbrace  x, f(x), \ldots, f^{10}(x)\right\rbrace \text{ où } x = -17.002728\ldots, \\
 \notag & \B_{4} := \left\lbrace  x, f(x), \ldots, f^{10}(x)\right\rbrace \text{ où } x = -16.999991\ldots.. \\
\label{prop_schwartz_neg}
& \text{La dérivée Schwartzienne de $f$ n'est pas partout négative sur } \mathbb{R}^{-}. \\
\label{prop_int_inv_neg}
& \text{Les intervalles $[-1, 0]$ et $[\nu_{1}, -1]$ sont invariants par $f$}. \\
\label{prop_int_cycles_neg}
& \text{Tout cycle d'entiers négatifs est répulsif.} \\
\label{prop_unbound_neg}
& \text{Il existe des orbites monotones non-bornées sur }  \mathbb{R}^{-}. 
\end{align}

\begin{proof}(indications)
\begin{flalign*}
\intertext{Propriété \eqref{prop_cycles_neg}: Les vitesses d'attraction sont données dans le tableau \ref{tab_cyc_neg}.}
\intertext{Propriété \eqref{prop_schwartz_neg}: La dérivée Schwartzienne est positive sur un intervalle contenant le point -0.2. On a en effet $Sf(-0.2) = 39.961\ldots$, où}
Sf(x) = \frac{f'''(x)}{f'(x)} - \frac{3}{2}\left( \frac{f''(x)}{f'(x)}\right) ^{2}.
\intertext{Propriété \eqref{prop_int_inv_neg}: La fonction $f$ est strictement décroissante sur l'intervalle $[\nu_{1}, 0]$ contenant le point fixe répulsif -1.}
\intertext{Propriété \eqref{prop_int_cycles_neg}: Voir les indications dans \cite[p.16]{Cha96}.}
\intertext{Propriété \eqref{prop_unbound_neg}: La démonstration est similaire à celle de \eqref{prop_unbound}.}
\end{flalign*}
\end{proof}

\begin{remarque}
Les cycles $\B_{2}$ et $\B_{4}$ sont très faiblement attractifs car leur multiplicateur est proche de 1 (cf. tableau \ref{tab_cyc_neg}). On vérifie également que les cycles contenant les points -5 et -17 sont très faiblement répulsifs, avec pour multiplicateurs respectifs les rationnels 9/8 et 2187/2048.
\end{remarque}

\begin{table}
\begin{center}
\begin{tabular}{crl}
 Point ou cycle attractif & Période & Multiplicateur \\
 \hline
 0 & 1 & 0.5 \\
 $\nu_{1}$ & 1 & $0.385708\ldots$ \\
 $\B_{1}$ & 3 & $0.036389\ldots$ \\
 $\B_{2}$ & 3 & $0.866135\ldots$ \\
 $\B_{3}$ & 11 & $0.003773\ldots$ \\
 $\B_{4}$ & 11 & $0.926287\ldots$ \\
 \hline
\end{tabular}
\caption{Coefficients multiplicateurs des points et cycles attractifs sur les réels négatifs.}
\label{tab_cyc_neg}
\end{center}
\end{table}

Comme précédemment, on note $c_{n}$ les points critiques proches des entiers $n < 0$, et on peut montrer que les itérés successifs de $c_{n}$ pour $n$ impair négatif restent proches des itérés de $n$, par valeurs inférieures. Nous avons vérifié numériquement pour tout entier $n$, $-1000 < n <0$, que 
\begin{itemize}
\item si $n$ est impair et $f^{k}(n) = -1$ (resp. -5, -17) pour un entier $k$, alors l'orbite  de $c_{n}$ converge vers $\nu_{1}$ (resp. $\B_{1}$, $\B_{3}$);
\item si $n$ est pair et $f^{k}(n) = -1$ (resp. -5, -17) pour un entier $k$, alors l'orbite  de $c_{n}$ converge vers $0$ (resp. $\B_{2}$, $\B_{4}$), sauf pour $n$=-34, -66, -98, -130, -132, -162, -174, -194, -202, -226, $\ldots$ où l'orbite  de $c_{n}$ converge vers $\B_{3}$, $\B_{1}$, $\B_{3}$, $\nu_{1}$, $\B_{3}$, $\B_{1}$, $\nu_{1}$, $\B_{1}$, $\nu_{1}$, $\nu_{1}$, $\ldots$ respectivement. On note que les entiers $n \equiv -2 \pmod{32}$ semblent toujours faire partie des exceptions.
\end{itemize}

Le plus souvent, lorsque $n<0$ est pair, l'orbite de $c_{n}$ reste proche de l'orbite de $n$, par valeurs supérieures. Pour $n$=-34, -98, -132, -162, -202, $\ldots$ les itérés de $c_{n}$ finissent pas être inférieurs aux itérés de $n$, sans s'en éloigner pour autant. Pour $n$=-66, -130, -174, -194, -258, $\ldots$ les orbites de $n$ et de $c_{n}$ sont décorrélées après un nombre fini d'itérations. Dans ce dernier cas, on observe une répartition des orbites de $c_{n}$ dans chacun des six bassins d'attraction de $\mathbb{R}^{-}$ : 0, $\nu_{1}$, $\B_{1}$, $\B_{2}$, $\B_{3}$ et $\B_{4}$.

\begin{conjecture}\label{conj_odd_critical_neg} \textbf{Points critiques d'ordre négatif impair}\\
Les points critiques $c_{n}$ sont attirés soit par le point fixe $\nu_{1}$, soit par l'un des cycles $\B_{1}$ ou $\B_{3}$, pour tout entier $n < 0$ impair.
\end{conjecture}

\section{Dynamique asymptotique}
\label{asymptotic}

Dans cette partie, nous étudions le comportement moyen de séquences finies d'itérations $S = \lbrace f^{i}(x) \rbrace_{i=0}^{n}$ telles que $\min \lbrace |f^{i}(x)| \rbrace_{i=0}^{n-1}  \gg 1$, afin de déterminer la vitesse moyenne asymptotique (i.e.\ au voisinage de $\pm\infty$).

Nous dirons de manière informelle que $S$ est  uniformément distribuée modulo 2 (u. d. mod 2) si et seulement si la discrépance à l'origine de $\lbrace f^{i}(x) \bmod 2 \rbrace_{i=0}^{n-1}$ dans l'intervalle $[0,2]$, notée $D_{n}^{*}(S \bmod 2)$, vérifie $D_{n}^{*}(S \bmod 2) \ll 1$.\footnote{On note $x \bmod 2$ la valeur modulo 2 de tout réel $x$, définie par $x \bmod 2 := x - 2 \lfloor \frac{x}{2} \rfloor$.}

On rappelle que la notion de discrépance est une mesure de l'uniformité de la distribution d'une séquence de points $\mathcal{X} = \lbrace  x_{1}, \ldots, x_{n} \rbrace  \in [a,b]^{n}$ et est définie par 
\begin{equation}\label{discrep}
D_{n}^{*}(\mathcal{X}) := \sup_{a \leq c \leq b } \left| \frac{ \left| \lbrace  x_{1}, \ldots ,  x_{n} \rbrace \cap [a,c)  \right|}{n}  - \frac{c-a}{b-a} \right|
\end{equation}

Elle intervient notamment dans l'inégalité de Koksma \cite{Kui74} :
\begin{theoreme} (Koksma)
Soit $f$ : $[a,b] \rightarrow \mathbb{R}$ une fonction à variation (totale) $V(f)$ bornée. Alors pour toute séquence $\mathcal{X} = \lbrace  x_{1}, \ldots, x_{n}\rbrace  \in [a,b]^{n}$, on a
$$ \left| \frac{1}{n}  \sum_{i=1}^{n} f(x_{i}) - \frac{1}{b-a}\int_{a}^{b} f(t) \, \mathrm{dt} \right| < V(f) D_{n}^{*}(\mathcal{X})$$
\end{theoreme}

Nous considérons dorénavant que la fonction $f$ définie par \eqref{eq_ext} s'applique sur $ \mathbb{R}$ tout entier. Comme $f$ ne s'annule qu'en $0$, il suit que $f^{n}(x)$ est de même signe que $x$ pour tout réel $x \neq 0$ et tout entier $n$.

Notre approche consiste à approximer $ f(x)/x $ par son asymptote sinusoïdale
\begin{equation}
g(x):=1-\frac{\cos(\pi x)}{2}
\end{equation}
 dont on détermine la moyenne géométrique.
\begin{lemme}\label{lem_geom}
La moyenne géométrique $\tau$ de la fonction réelle $ g(x) =  1-\cos(\pi x)/2$ sur $[0, 2]$ est égale à $\alpha/4$, où $\alpha = 2+\sqrt{3} $ est racine du polynôme $X^{2} - 4X + 1$.
\end{lemme}
\begin{proof}
On cherche à calculer $\tau := \exp\left(\frac{1}{2}\int_{0}^{2}\ln \left( g(t)\right)  \, \mathrm{dt} \right) $
avec 
$$g(t) = 1-\cos(\pi t)/2 = \left( \alpha - e^{i \pi t} \right)\left( \alpha - e^{-i \pi t} \right) / (4 \alpha)  = \left| \alpha - e^{i \pi t} \right|^{2} / (4 \alpha).$$
 On obtient  $$ \ln  \tau  = \int_{0}^{2} \ln \left| \alpha - e^{i \pi t} \right|  \, \mathrm{dt} - \ln\left( 4 \alpha\right).$$
 
La formule de Jensen relative aux fonctions analytiques sur le disque de centre $\alpha$ et de rayon 1 donne le résultat attendu 
$$\ln  \tau  = 2 \ln\alpha - \ln ( 4 \alpha ) = \ln \left(  \frac{\alpha}{4} \right) .$$
\end{proof}
On montre à présent qu'au voisinage de $\pm \infty$ toute séquence d'itérations u. d. mod 2 de $f$ décroit avec une vitesse moyenne proche de $ \tau = ( 2+\sqrt{3} ) / 4 \simeq 0.933 \ldots $.

\begin{theoreme}\label{th_estim}
Soit une séquence finie d'itérations $S = \lbrace f^{i}(x) \rbrace_{i=0}^{n}$ telle que $\min \lbrace |f^{i}(x)| \rbrace_{i=0}^{n-1} \geq M$ pour un réel $M > \frac{1}{3}$.
Alors on a\\
$$ \left| \frac{1}{n} \ln \left( \frac{f^{n}(x)}{x} \right)  - \ln \tau \right| < 2 \left( \ln 3\right)  D_{n}^{*}(S \bmod 2) - \ln \left( 1 - \frac{1}{3M} \right) .$$
\end{theoreme}
\begin{proof}
On considère la formulation $ f(t) = g(t) \left( t + h(t) \right)$ où $h$ est la fonction périodique $$h(t) := \frac{1 - \cos (\pi t)}{4 g(t)} = \frac{1 - \cos (\pi t)}{4-2\cos (\pi t)}.$$
On a donc $$\frac{f^{n}(x)}{x} = \prod_{i=0}^{n-1} \frac{f^{i+1}(x)}{f^{i}(x)} 
= \prod_{i=0}^{n-1} g \left( f^{i}(x) \right) \left( 1 + \frac{h \left( f^{i}(x) \right)}{f^{i}(x)} \right) $$
Il vient alors $$ \frac{1}{n} \ln \left( \frac{f^{n}(x)}{x} \right) - \ln  \tau  = 
 A + B$$
avec
$$A = \frac{1}{n}  \sum_{i=0}^{n-1} \ln \left( g \left( f^{i}(x) \right) \right) - \ln \tau$$
et 
$$B = \frac{1}{n}  \sum_{i=0}^{n-1} \ln  \left( 1 + \frac{h \left( f^{i}(x) \right)}{f^{i}(x)} \right) $$
D'après le lemme \ref{lem_geom}, 
$$ \ln \tau  = \frac{1}{2}\int_{0}^{2}\ln \left( g(t)\right)  \, \mathrm{dt}$$
On applique l'inégalité de Koksma :
$$ \vert A \vert \le V(\phi) D_{n}^{*}(S \bmod 2)$$ où $V(\phi)$ est la variation totale de la fonction $\phi(t) := \ln \left( g(t) \right)$ sur $[0,2]$, soit $ V(\phi) = 2 \phi(1) - \phi(2) - \phi(0) = 2 \ln 3$.\\

Pour majorer $|B|$, on vérifie que la fonction $h(t)$ est à valeur dans $[0,1/3]$ avec un maximum en $t=1$. On en déduit que $$|B| \le \max \left(  -\ln  \left( 1 - \frac{1}{3 M} \right) , \: \ln  \left( 1 + \frac{1}{3 M} \right) \right) = -\ln  \left( 1 - \frac{1}{3 M} \right).$$
\end{proof}

Le théorème \ref{th_estim} est inopérant pour les séquences d'entiers, dont la vitesse moyenne attendue est $\sqrt{3} / 2$, strictement inférieure à $\tau$. Il permet toutefois d'établir un lien entre la vitesse moyenne et la distribution modulo 2 des itérations.

\begin{theoreme}\label{th_liminf}
Soit $x$ un réel d'orbite $\lbrace f^{i}(x) \rbrace_{i=0}^{\infty}$ telle que 
$$\liminf_{i \to \infty} |f^{i}(x)| > \frac{1}{3(1 - \tau)} \simeq 4.97\ldots $$
Alors l'orbite de $x$ n'est pas uniformément distribuée modulo 2.
\end{theoreme}
\begin{proof}
Il existe un entier positif $N$ et un réel $a>1$ tels que $$|f^{i}(x)| \geq \frac{a}{3(1 - \tau)}$$ pour tout $i \geq N$.\\

On considère les séquences finies $S_{n} = \lbrace f^{i}(x)\rbrace_{i=N}^{n+N}$ pour tout $n$ entier positif, et on pose $M_{n} := \min \lbrace |f^{i}(x)| \rbrace_{i=N}^{n+N}$.\\

D'après le théorème \ref{th_estim},
$$  \frac{1}{n} \ln \left( \frac{f^{n+N}(x)}{f^{N}(x)} \right)  - \ln \tau  < 2 (\ln 3) D_{n}^{*}(S_{n} \bmod 2) - \ln \left( 1 - \frac{1}{3M_{n}} \right) .$$

Il vient
$$ 2 (\ln 3)  D_{n}^{*}(S_{n} \bmod 2) > A_{n} + B_{n}$$
avec
$$ A_{n} = \frac{1}{n} \ln \left( \frac{f^{n+N}(x)}{f^{N}(x)} \right)$$
et
$$ B_{n} = - \ln \tau + \ln \left( 1 - \frac{1}{3M_{n}} \right).$$

D'une part, on vérifie aisément que $ \liminf_{n \to \infty} A_{n} \geq 0$. D'autre part, on a
$$B_{n} \geq - \ln \tau + \ln \left( 1 - \frac{1 - \tau}{a} \right) = \ln \left( 1 + \frac{(a-1)(1-\tau)}{a \tau} \right) > 0.$$

On obtient donc le résultat souhaité :
$$ \liminf_{n \to \infty} D_{n}^{*}(S_{n} \bmod 2) \geq \frac{\ln \left( 1 + \frac{(a-1)(1-\tau)}{a \tau} \right)}{2\ln 3} > 0$$
\end{proof}

L'existence d'orbites tendant vers l'infini a été prouvée par Chamberland pour la fonction $f$ et le corollaire \ref{cor_div} donne une condition nécessaire sur l'ensemble des valeurs modulo 2 d'une telle orbite.

\begin{corollaire}\label{cor_div}
Soit $x$ un réel d'orbite $\lbrace f^{i}(x) \rbrace_{i=0}^{\infty}$ divergente telle que 
$$\lim_{i \to \infty} |f^{i}(x)| = +\infty.$$
Alors l'orbite de $x$ n'est pas u. d. mod 2.
\end{corollaire}

\begin{remarque}
 Ce résultat renforce la conjecture \ref{conj_unbound}. En effet, on peut s'attendre à ce que la condition de distribution uniforme modulo 2 des itérations de $f$ soit le plus souvent valide au voisinage de $\pm\infty$, compte tenu des propriétés suivantes :
 \begin{itemize}
\item le diamètre et la densité des zones contractantes tend vers 0,
\item l'amplitude des oscillations devient infiniment grande.
\end{itemize}
\end{remarque}


\end{document}